\documentclass[psamsfonts, reqno, 12pt]{amsart}
\usepackage{amsfonts,amsmath, amsthm, amssymb, latexsym, epsfig}
\usepackage[all]{xy}
\usepackage{xspace}
\usepackage[mathscr]{eucal}

	\advance \topmargin by -\headheight
	\advance \topmargin by -\headsep

	\evensidemargin \oddsidemargin
\RequirePackage{amsfonts}
\RequirePackage{amsmath}
\RequirePackage{amsthm}
\RequirePackage{amssymb}
\RequirePackage{latexsym}
\RequirePackage{epsfig}
\RequirePackage{verbatim}
\RequirePackage[mathscr]{eucal}
\RequirePackage[all]{xy}

	\newcommand{\Aut}{\ensuremath{\operatorname{Aut}}}
	\newcommand{\Ker}{\ensuremath{\operatorname{Ker}}}

	\newcommand{\thin}{ \mbox{H}^2(S, D)}

	\def\Aut{\mathop{\rm Aut}}
	\def\Out{\mathop{\rm Out}}
	\def\Inn{\mathop{\rm Inn}}
	\def\Sum{\mathop{\rm Sum}}
	\def\Cen{\mathop{\rm Cen}}
	\def\Stab{\mathop{\rm Stab}}

        \theoremstyle{plain}
	\newtheorem{thm}{Theorem}[section]
	\newtheorem{lemma}[thm]{Lemma}
	\newtheorem{theorem}[thm]{Theorem}
	\newtheorem{proposition}[thm]{Proposition}
	\newtheorem{corollary}[thm]{Corollary}

	\theoremstyle{definition}
	\newtheorem{definition}[thm]{Definition}

	\numberwithin{equation}{section}
	\numberwithin{figure}{section}

\begin{document}
	\title{Square-Free Rings with Local Units}
	\author{Martin Montgomery}
        \address{Department of Mathematics and Statistics \\
       Sam Houston State University}
        \email{mmontgomery@shsu.edu}
        \date{\today}
	

	\begin{abstract}
Recently, the author characterized all artinian square-free rings with identity.  Here, those results are extended to the setting of rings with local units and a square-free ring (not necessarily with identity) is characterized by a square-free semigroup $S$, a division ring $D$, and a 2-cocycle $(\alpha, \xi)$ from the non-abelian cohomology of $S$ with coefficients in $D$.  We use this characterization of square-free rings to determine, in cohomological terms, exactly when a square-free ring has an associated square-free algebra structure.  Finally, using our characterization of a square-free ring $R$, it follows there is a short exact sequence
\begin{align*}
1 \longrightarrow H^1_{(\alpha, \xi)} (S, D) \longrightarrow  \mbox{Out }R
\longrightarrow W \longrightarrow 1
\end{align*}
where $W$ is the stabilizer of the action of $(\alpha, \xi)$ on $\Aut{S}$, and when $(\alpha, \xi)$ is trivial $W=\mbox{Aut}(S)$ and the sequence splits.  	
	\end{abstract}
        \maketitle

\section*{Introduction}
 
The study of outer automorphisms of certain algebraic structures, beginning with \cite{Stanley}, eventually led to the characterization of the outer automorphism group of finite dimensional square-free algebras by Anderson and D'Ambrosia \cite{And-D'Amb}.  A finite dimensional algebra $A$ over a field $K$ is \textit{square-free} if
$$\mbox{dim}_K (e_iAe_j) \leq 1$$
for every pair $e_i, e_j \in A$ of primitive idempotents.  In \cite{And-D'Amb2} Anderson and D'Ambrosia extended the characterization of the outer automorphism group to arbitrary square-free algebras.  Such an algebra $A$ is a \textit{ring with local units}, that is, although $A$ need not contain an identity element, for every pair $x, y \in A$ there exists an idempotent $u \in A$ such that $x, y \in uAu$.

Square-free rings were defined by D'Ambrosia \cite{D'Ambrosia} as a generalization of square-free algebras.  An artinian ring $R$ is \textit{square-free} if each indecomposable projective $R$-module has no repeated composition factors.  In \cite{D'Ambrosia}, the basic properties of square-free rings were developed and examples were provided of square-free rings that are not $D$-algebras, i.e., rings of the form $D \otimes_K A$, where $D$ is a division ring with center $K$ and $A$ is a square-free $K$-algebra.  Therefore, the characterization of a square-free ring does not follow automatically from a similar characterization for a square-free algebra.

The purpose of this note is to unify the results of \cite{And-D'Amb}, \cite{And-D'Amb2}, and \cite{Montgomery1} as a general theory for square-free rings.  As was shown in previous special cases, for a square-free ring $R$ (not necessarily with identity) there is a canonical square-free semigroup $S$, a division ring $D$, and a cocycle $(\alpha, \xi) \in Z^2(S, D)$ such that $R$ may be characterized  as the $D$-vector space with basis the nonzero elements of $S$ and with multiplication ``twisted" by the element $(\alpha, \xi)$.  From this characterization, we can determine the necessary and sufficient conditions under which a square-free ring $R$ with associated cocycle $(\alpha, \xi) \in Z^2(S, D)$ is a $D$-algebra. 

Finally, adopting the notion of inner automorphism from \cite{And-D'Amb2}, we use our characterization of square-free rings to generalize the main result of \cite{Montgomery1}. Thus, given a square-free ring $R=D_{\xi}^{\alpha}S$, we prove there is a short exact sequence 
\begin{align*}
1 \longrightarrow H^1_{(\alpha, \xi)} (S, D) \longrightarrow  \mbox{Out }R
\longrightarrow W \longrightarrow 1
\end{align*}
where $W$ is the stabilizer of the action of $(\alpha, \xi)$ on $\Aut(S)$, and when $(\alpha, \xi)$ is trivial $W=\mbox{Aut}(S)$ and the sequence splits.

\section{Semigroup Cohomology}\label{cohomology}

Let $S$ be a (not necessarily finite) semigroup with zero $\theta$ and set $E$ of nonzero pairwise orthogonal idempotents such that
$$S=\bigcup_{e_i, e_j \in E} e_i \cdot S \cdot e_j.$$
We say $S$ is \textit{square-free} if
$$|e_i\cdot S \cdot e_j\setminus \theta | \leq 1$$
for each $e_i, e_j \in S$, and we let $S^*=S \setminus \{ \theta \}$.  We write $s_{ii}=e_i \in E$ and $s_{ij} \in S^*$ whenever $s_{ij}=e_i \cdot s_{ij} \cdot e_j$.  Since $S$ is square-free this notation is unambiguous.

In \cite{And-D'Amb}, a cohomology with coefficient set the units of a field was used with square-free semigroups.  Let $S^{<0>}=E$ and for all $n \geq 1$ define $S^{<n>}$ by
$$S^{<n>}=\{ (s_1, s_2, \ldots, s_n) |~s_1 \cdot s_2 \cdots s_n \neq \theta \}.$$
For $K$ a field with nonzero elements $K^*$ we consider the (abelian) groups $F^n (S, K)$ consisting of all set maps from $S^{<n>}$ to $K^*$ under the operation of pointwise multiplication.

Define a boundary map $\partial^n : F^n (S, K) \longrightarrow F^{n+1} (S, K)$ by
$$ (\partial^0 \phi)(s_{ij})=\phi(e_j) \phi(e_i)^{-1}$$
and for $n \geq 1$
\begin{align*}
(\partial&^n  \phi)(s_1, s_2, \ldots s_{n+1})\\ 
&=
\phi(s_2, \ldots, s_{n+1}) \left( \prod_{i=1}^n \phi(s_1, \ldots, s_i \cdot s_{i+1}, \ldots, s_{n+1})^{(-1)^i} \right)
\phi(s_1, \ldots, s_n)^{(-1)^{n+1}}.
\end{align*}
As usual, $\partial^{n+1} \partial^n=1$ and we define
$$Z^n (S, K)=\mbox{Ker } \partial^n, ~~\mbox{and}~~B^n(S, K)=\mbox{Im } \partial^{n-1}$$
as the groups of $n$-dimensional cocycles and coboundaries, respectively.

Of particular interest are elements $\zeta \in Z^2 (S, K)$.  From the equation above, these satisfy
\begin{equation}\label{abelian 2-cocycle}
\zeta(s_{jk}, s_{k \ell} ) \zeta (s_{ij}, s_{jk} \cdot s_{k \ell})  =\zeta(s_{ij}, s_{jk} )  \zeta(s_{ij} \cdot s_{jk}, s_{k \ell})
\end{equation}
whenever $s_{ij} \cdot s_{jk} \cdot s_{k\ell} \neq \theta$.  
We now consider a non-abelian cohomology that generalizes Equation \ref{abelian 2-cocycle}.

Let $D$ be a division ring with set of units $D^*=D\setminus \{0 \}$ and let $\mbox{Aut}(D)$ be the group of ring automorphisms of $D$.  Define $S^{<n>}$ for each $n \geq 0$ as before and let $F^n(S, D)$ and $F^n(S, \Aut(D))$ (respectively) denote the set of all functions from $S^{<n>}$ to $D$ and $\Aut(D)$ (respectively).  Again, these sets $F^n(S, D)$ are groups with operation multiplication of functions, but need no longer be abelian.  For $\eta \in F^1(S, D)$ we denote by $\eta(s_{ij})$ the image of $s_{ij}$ under $\eta$ in $D^*$.  The sets $F^n(S, \Aut(D))$ are groups under composition of functions.  For $\mu \in F^0(S, \Aut(D))$, we denote by $\mu_i$ the automorphism of $D$ that is the image of $e_i \in E$ under $\mu$.
For $\alpha \in F^1(S, \Aut(D))$, we denote by $\alpha_{ij}$ the automorphism of $D$ that is the image of $s_{ij} \in S^*$ under $\alpha$.  For $s_{ii}=e_i$, we abbreviate $\alpha_{ii}=\alpha_i$.

There is a group action $\square$ of $F^0 (S, \Aut(D))$ on $F^1(S, D)$ defined as follows.  For $\mu \in F^0(S, \Aut(D))$, $\eta \in F^1(S, D^*)$, and $s_{ij} \in S^*$ we have
$$(\mu \square \eta)(s_{ij})=\mu_i( \eta(s_{ij})).$$
As in \cite{Montgomery1}, define an operation on the set
$$F^0(S, \Aut (D)) \times F^1 (S, D)$$
with multiplication given by 
$$(\hat{\mu}, \hat{\eta})(\mu, \eta)=( \mu \hat{\mu}, [\mu \square \hat{\eta}] \eta).$$
A straightforward check shows $F^0(S, \Aut (D)) \times F^1 (S, D)$ with the above operation is a group, which we denote by $G(S, D)$.

For $0 \neq d \in D$, denote by $\tau_d \in \Aut(D)$ the inner automorphism
$$\tau_d: x \longrightarrow dxd^{-1}.$$

Let $(\alpha, \xi) \in F^1(S, \Aut(D)) \times F^2(S, D)$.  If for all $(s_{ij}, s_{jk}, s_{k \ell} ) \in S^{<3>}$ the pair $(\alpha, \xi) $ satisfies the \textit{cocycle identities:}
	\begin{align}\label{2cochain}
	\alpha_{ij} (\xi(s_{jk}, s_{k \ell}) )~ \xi(s_{ij}, s_{jk} \cdot s_{k \ell})=\xi(s_{ij}, s_{jk}) \xi(s_{ij} \cdot s_{jk}, s_{k \ell})
	\end{align}
and
	\begin{align} \label{twistauto}
	\alpha_{ij} \circ \alpha_{jk}=\tau_{\xi(s_{ij}, s_{jk})} \circ \alpha _{ik}
	\end{align}
then $(\alpha, \xi)$ is called a \textit{2-cocycle}.  We denote the set of all such 2--cocycles by $Z^2(S, D)$.  Note that if $\alpha_{ij}=1$ is the identity automorphism for all $s_{ij} \in S^*$, then Equation \ref{2cochain} reduces to Equation \ref{abelian 2-cocycle} and Equation \ref{twistauto} implies that $\xi(s_{ij}, s_{jk} ) \in \Cen(D)$ for all $(s_{ij}, s_{jk}) \in S^{<2>}$.  It is easy to show that for $K=\Cen(D)$ with $\zeta \in Z^2(S, K)$, $(1, \zeta)$ satisfies Equations \ref{2cochain} and \ref{twistauto}.  Hence, we may identify $Z^2(S,K)$ as a subgroup of $Z^2(S,D)$.

By Theorem 1.3 of \cite{Montgomery1}, there is a group action $*$ of  $G(S, D)$ on $Z^2 (S, D)$ defined by 
$$(\mu, \eta) * (\alpha, \xi)=(\beta, \zeta)$$
where
\begin{equation}
\beta_{ij}=\mu_i^{-1} \circ \tau_{\eta(s_{ij})} \circ \alpha_{ij} \circ \mu_j.
\end{equation}
and 
\begin{equation}
\zeta(s_{ij},s_{jk})= \mu_i^{-1} [ \eta(s_{ij}) \alpha_{ij} (\eta(s_{jk})) \xi (s_{ij}, s_{jk}) \eta( s_{ij} \cdot s_{jk})^{-1}]
\end{equation}
for $(\mu, \eta) \in G(S, D)$, $(\alpha, \xi) \in Z^2(S,D)$ and $(s_{ij},s_{jk})\in S^{<2>}$.

\begin{definition}  We say that two 2-cocycles $(\alpha, \xi)$, $(\beta, \zeta) \in Z^2(S, D)$ are \textit{cohomologous} if there is a $(\mu, \eta) \in G(S,D)$ with
$$(\mu, \eta)*(\alpha, \xi)=(\beta, \zeta).$$
\end{definition}
Thus, $(\alpha, \xi)$, $(\beta, \zeta) \in Z^2(S, D)$ are cohomologous iff they lie in the same orbit under the $*$-action of $G(S,D)$, and in this case we abbreviate $[\alpha, \xi]=[\beta, \zeta]$.  As in \cite{Dedecker}, we call the orbits under $*$ the \textit{thin 2-cohomology} and denote them by
$$\thin.$$
Note then that two 2-cocycles $(\alpha, \xi)$, $(\beta, \zeta)$ are cohomologous iff there is some $(\mu, \eta) \in G(S,D)$ with

\begin{equation}\label{stuff1}
\mu_i \circ \beta_{ij} \circ \mu_j^{-1} =\tau_{\eta(s_{ij})} \circ \alpha_{ij} 
\end{equation}
and
\begin{equation}\label{stuff2}
\mu_i[\zeta(s_{ij}, s_{jk})]=\eta(s_{ij}) \alpha_{ij} (\eta(s_{jk})) \xi (s_{ij}, s_{jk}) \eta( s_{ij} \cdot s_{jk})^{-1}.
\end{equation}

 \begin{definition} We say a cocycle $(\alpha, \xi) \in Z^2 (S, D)$ is \textit{normal} if $\xi( e_i, e_i)=1$ for every $e_i \in E$.  
\end{definition}
In general, since $(\alpha, \xi) \in Z^2(S, D)$ satisfies Equation \ref{twistauto}, we have
 $$\alpha_{i} \circ \alpha_{i} =\tau_{\xi(e_i,e_i)} \circ \alpha_{i} \Rightarrow \alpha_{i}=\tau_{\xi(e_i,e_i)}$$
 for $e_i=e_i \cdot e_i \in E$.  So if $(\alpha, \xi)$ is normal, we must have $\alpha_i$ the identity automorphism on $D$ for all $e_i \in E$.  It was shown in \cite{Montgomery1} that every cohomology class $[\alpha, \xi] \in \thin$ has a normal representative.
 
In \cite{And-D'Amb2},  an algebra was built using a (possibly infinite) square-free semigroup $S$, a field $K$, and $\zeta \in Z^2(S, K)$.  A similar construction is possible using a division ring $D$ and $(\alpha, \xi) \in Z^2(S, D)$.

\begin{definition}
Given a division ring $D$, (possibly infinite) square-free semigroup $S$, and 2-cocycle $(\alpha, \xi) \in Z^2(S,D)$, let $R$ be the left $D$-vector space defined on $S^*=S\setminus \{\theta\}$ with multiplication defined on the basis $S^*$ by
$$s_{ij} s_{k\ell}=
\begin{cases}  \xi(s_{ij}, s_{k\ell}) s_{i\ell}  & \quad \mbox{if} \quad s_{ij} \cdot s_{k\ell} \neq \theta \vspace{.2cm}   \cr  0 & \quad \mbox{if} \quad s_{ij} \cdot s_{k\ell} = \theta  
\end{cases} $$
for all $s_{ij}, s_{jk} \in S$, and 
$$s_{ij}d=\alpha_{ij}(d)s_{ij}$$
for all $s_{ij} \in S$ and $d \in D^*$.  Then it follows from Equations \ref{2cochain} and \ref{twistauto} that this multiplication extended linearly is associative and produces a ring, that we denote by
$$D^{\alpha}_{\xi} S.$$
When $D=K$ is a field and $(1, \zeta) \in Z^2(S, K) \subseteq Z^2(S,D)$, we write
$$K_{\xi} S$$
for the $K$-algebras described in \cite{And-D'Amb2}.
\end{definition}

For an automorphism  (written as a right operator) $\phi \in \Aut (S)$ and $s \in S$, we write $s^{\phi}=(s)\phi$.  We can use the group $\mbox{Aut} (S)$ of semigroup automorphisms to define another action on $Z^2 (S, D)$.  Given $\phi \in \mbox{Aut} (S)$ and 2-cocycle $(\alpha, \xi) \in Z^2 (S, D)$, define a new 2-cocycle $(\alpha^{\phi}, \xi^{\phi}) \in Z^2(S, D)$ by
$$\alpha^{\phi} (s_{ij})=\alpha (s_{ij}^{\phi})$$
for all $s_{ij}\in S$ and 
$$ \xi^{\phi} (s_{ij}, s_{jk})=\xi( s_{ij}^{\phi}, s_{jk}^{\phi})$$
for all $(s_{ij}, s_{jk}) \in S^{<2>}$.  
It follows $[\alpha, \xi]=[\beta, \zeta]$ if and only if $[\alpha^{\phi}, \xi^{\phi}]=[\beta^{\phi}, \zeta^{\phi}]$ and $\mbox{Aut}(S)$ acts on $\thin$.

An easy extension of Lemma 1.3 of \cite{Montgomery1}  gives the following:

\begin{lemma} \label{rightiso}
Let $S$ be a square-free semigroup, and let $[\alpha, \xi], [\beta, \zeta] \in H^2(S, D)$.  If 
	$$[\beta, \zeta] = [\alpha^{\phi}, \xi^{\phi}]$$
for some $\phi \in \mbox{Aut }S$, then as rings, we have
	$$D^{\beta}_{\zeta} S \cong D^{\alpha}_{\xi}S.$$
It follows then, for each $(\alpha, \xi) \in Z^2(S, D^*)$ there exists a normal $(\beta, \zeta) \in Z^2(S, D)$ with $D^{\beta}_{\zeta} S \cong D^{\alpha}_{\xi}S$.  
\end{lemma}
 
 In light of Lemma \ref{rightiso}, we may assume that a ring $D^{\alpha}_{\xi} S$ is written using a normal representative $[\alpha, \xi] \in H^2 (S, D)$.  Then as elements of the ring $D^{\alpha}_{\xi} S$, $e^2=\xi(e, e) e \cdot e =e$ and it follows that $E$ is a complete set of primitive idempotents of $D^{\alpha}_{\xi} S$.

 \section{Characterization of Square-Free Rings With Local Units}
 
 
 In this section, we introduce the notion of square-free rings with local units.  We begin by showing such rings have many properties in common with square-free artinian rings, concluding with Theorem \ref{leftiso}, which shows a square-free ring with local units can be characterized as $D^{\alpha}_{\xi}S$, where $D$ is a division ring, $S$ is a square-free semigroup, and $[\alpha, \xi] \in H^2(S, D)$.  En route, we give a summary of Section 1 of \cite{And-D'Amb2}, which establishes the notion of inner automorphism for rings with local units.  Finally, we conclude with a Corollary \ref{algebra}, which gives necessary and sufficient conditions for a square-free ring to have a related square-free algebra structure.
 
 Let $R$ be a ring, not necessarily with an identity element, with Jacobson radical $J=J(R)$.  A set $U$ of idempotents in $R$ is a \textit{set of local units} for $R$ if for each pair $x, y \in R$ there exists $u \in U$ with $x, y \in uRu$.  Let $E$ be a set of primitive pairwise orthogonal idempotents in $R$ and denote by $\mathscr{U}_E$ the set of all finite sums of orthogonal sets from $E$.  So
$$\mathscr{U}_E=\{e_1+ \cdots e_n:~e_1, \ldots, e_n \in E~\mbox{ are orthogonal},~n=1,2,3, \ldots \}.$$

If the set $\mathscr{U}_E$ is a set of local units for $R$, then we say that $\mathscr{U}_E$ is an \textit{atomic} set of local units with set of \textit{atoms} $E$.  Not every set of local units is atomic and for a ring with an atomic set of local units there need not be a unique set of atoms.  For a ring $R$ with atomic set $E$, the set $\{uRu|~u \in \mathscr{U}_E\}$ is a directed set of unital subrings of $R$, directed by containment, with
$$R=\bigcup_{u \in \mathscr{U}_E } uRu.$$
If $R$ has atoms $E$, then for each $e_i \in E$ the left and right modules $Re_i$ and $e_iR$ are indecomposable, projective, with
$${}_{R}R=\bigoplus_{e_i \in E} Re_i \hspace{1cm} \mbox{and} \hspace{1cm} R_{R}=\bigoplus_{e_i \in E} e_iR.$$
We say $R$ is \textit{locally artinian} if the unital ring $uRu$ is artinian for each $u \in \mathscr{U}_E$.  


Suppose $M$ is a left finitely generated indecomposable $R$-module, where $R$ is a basic artinian ring with identity $1=e_1+e_2 + \ldots e_n$.  For each $e_i$, let $c_i(M)$ denote the number of composition factors isomorphic to $Re_i/Je_i$.  We say $M$ is \textit{square-free} if $c_i(M) \leq 1$ for each $e_i \in E$.  By lemma 1.2 of \cite{D'Ambrosia} an indecomposable left module $M$ is square-free if an only if for each $x \in M$ with $e_ix \neq 0$ we have
$$Re_ix=Re_iM.$$

\begin{definition}
Let $R$ be a locally artinian ring with set of atoms $E$.  Let $M$ be a finitely generated indecomposable left (right) $R$-module.  We say $M$ is left (right) square-free if for every $e_i \in E$, $Re_i x=Re_iM$ ($xe_iR=Me_iR$) whenever $x \in M$ with $e_i x \neq 0$ ($xe_i \neq 0$).
\end{definition}

As in the unital case, we'll say $M$ is \textit{square-free} if all indecomposable summands are square-free.  The ring $R$ is left (right) square-free if ${}_R R$ ($R_R$) is square-free.  We say $R$ is square-free if it is both left and right square-free.  It is easy to show that the algebras of \cite{And-D'Amb2} satisfy the definition above, so square-free algebras are square-free rings.  At the end of this section, we investigate the conditions under which a square-free ring has a related algebra structure.  

If $R$ is a locally artinian ring with atomic set $E$, for each $u \in \mathscr{U}_E$, $uRu$ is a unital ring and left $uRu$-modules are of the form $uM$, for some left $R$-module $M$.  If $M$ is square-free then $uM$ is square-free.  It follows then for a square-free ring $R$ the unital rings $uRu$ are square-free for all $u \in \mathscr{U}_E$.

We now have easy extensions of some results from \cite{D'Ambrosia}:

\begin{theorem}\label{square-free props}
Let $R$ be a locally artinian square-free ring with atomic set $E$ and $e_i, e_j \in E$.

\begin{itemize}
\item[(a)]  If $e_ire_j \neq 0$ for some $r \in R$, then $Re_ire_j=Re_iRe_j$

%

\item[(b)]  For each pair $e_i, e_j \in E$, $e_iRe_i$ is a division ring with $$\mbox{dim}_{e_iRe_i} (e_iRe_j) \leq 1.$$

\item[(c)]  For each $e_ixe_j \in e_iRe_j$, there is a ring isomorphism 
$$\alpha_{ij} : e_jRe_j \longrightarrow e_iRe_i$$
 given by
$$\alpha_{ij} (e_jre_j) \cdot e_ixe_j=e_ixe_j \cdot e_jre_j.$$

\item[(d)]  The set $S(R)=\{ e_iRe_j: ~e_i,e_j \in E\} \cup \{0 \}$ is a square-free semigroup with set of idempotents $E(R)=\{e_iRe_i:~e_i \in E\}$ and multiplication $(e_iRe_j) \cdot (e_kRe_{\ell})=e_iRe_je_kRe_{\ell}$.

\end{itemize}
\end{theorem}

\begin{proof}
The result (a) follows easily from the definition of square-free.  Suppose $Re_i/Je_i \cong Re_j/Je_j$.  Choose $u \in \mathscr{U}_E$ such that $e_i, e_j \in uRu$.  Then
$uRe_i/uJe_i \cong uRe_j/uJe_j$ are simple $uRu$-modules.  Since $R$ is square-free, the unital ring $uRu$ is square-free with square-free summand $uRe_j$.  Then $uRe_j/uJe_j$ is the only composition factor of the square-free module $uRe_j$ that is isomorphic to $uRe_i/uJe_i$, hence $e_i Je_j=0$.  Now (b), (c), and (d) follow by arguments similar to the unital cases (see Theorem 1.4,  Corollary 1.5, and Proposition 1.6 of \cite{D'Ambrosia}).  
\end{proof}

For a square-free ring $R$ the division ring $D$ of Theorem \ref{square-free props} parts (b) and (c) is called the \textit{division ring of $R$}.  The semigroup from Theorem \ref{square-free props} part (d) is called the \textit{regular semigroup of R}.  We call an injective map $\chi: S(R) \rightarrow R$ with $\chi(0)=0$, and with $\chi(e_iRe_j)=s_{ij} \in e_iRe_j$ where 
$$s_{ij} = \begin{cases} e_i  & \quad \mbox{if} \quad i=j \vspace{.2cm} 
\cr  s_{ij} \in e_iRe_j \setminus \{0 \}  & \quad  \mbox{if} \quad e_iRe_j \neq 0 \vspace{.2cm} 
\cr  0 & \quad  \mbox{if} \quad e_iRe_j =0 
\end{cases} $$
a \textit{choice map for S(R)}.  Given such a choice, $\chi$, we denote the image $\mbox{Im } \chi$,
$$S_{\chi}=\{s_{ij}=\chi(e_iRe_j)|~e_iRe_j \in S(R) \} \cup \{0 \},$$
and define a multiplication on $S_{\chi}$ by

$$s_{ij} \cdot s_{k \ell}= \begin{cases} s_{i \ell}  & \quad \mbox{if} \quad e_iRe_je_kRe_{\ell} \neq 0 \vspace{.2cm}   \cr  0 & \quad  \mbox{otherwise.} 
\end{cases} $$
Since $R$ is square-free, this multiplication is well-defined and $S_{\chi}$ is a semigroup with zero.  Since $s_{ii}=e_i$ is idempotent the set  $E$ is the set of nonzero idempotents of $S_{\chi}$ and $S_{\chi}$ is a square-free semigroup.  Since each choice map $\chi: S(R) \rightarrow R$ is injective, it follows for every choice map $\chi$, $S_{\chi} \cong S(R)$.  


For a square-free semigroup $S$ with idempotent set $E$, we can define a relation $\sim$ on $E$ by $e_i \sim e_j$ if and only if $e_i \in e_i S e_j Se_i$.  By Lemma 1.1 of \cite{And-D'Amb}, this is an equivalence relation.  If $\overline{E}$ is a complete set of representatives of the $\sim$-equivalence classes, then
$$\overline{S}=\bigcup_{\overline{e}_i, \overline{e}_j \in \overline{E} } \overline{e}_i S \overline{e}_j$$
is a square-free semigroup with idempotent set $\overline{E}$.  The semigroup $\overline{S}$ is called the \textit{reduced semigroup} of $S$ and is independent (up to isomorphism) of the choice of representatives of $\overline{E}$.

Let $e_i \in E$.  The set
$$S[e_i]=\bigcup_{e_j \sim s_k \sim s_i} e_j S e_k$$
is a square-free subsemigroup of $S$ on which $\sim$ is the universal relation.  We will call the subsemigroup $S[e_i]$ the $\sim$-block of $S$.  If the $\sim$-equivalence relation of $e_i$ has finite cardinality $n$, then it follows the block $S[e_i]$ is isomorphic to the semigroup
$$ \{ e_{ij} \in \mathbb{M}_n (\mathbb{Z})~|~1 \leq i, j \leq n \} \cup \{ 0 \}$$
of matrix units in the ring $\mathbb{M}_n (\mathbb{Z})$ of $n \times n$ matrices of integers together with the zero matrix.  We now have a result that shows that 2-cocyles restricted to the $\sim$-blocks are essentially trivial.

\begin{lemma}\label{reduced}
If $S$ is a square-free semigroup and $(\alpha, \xi) \in Z^2 (S, D)$, then there is $(\mu, \eta) \in G(S, D)$
such that for every $\sim$-block $S[e_i]$ of $S$, the restriction of $(\mu, \eta) *(\alpha, \xi)$ to $S[e_i]$ is $1$.
\end{lemma}

\begin{proof}
Let $e_i \in E$ and consider $s_{jk} \in S[e_i]$.  
By definition of the the $\sim$-block $S[e_i]$, there exists a unique $s_{ij}=e_i \cdot s_{ij} \cdot e_j \in e_iS[e_i]$ such that
$s_{ij} \cdot s_{jk} \neq \theta.$
For every $e_j \in S[e_i]$, define $\mu:E \rightarrow \Aut(D)$ by $\mu_j=\alpha_{ij}^{-1}$.

If $s_{jk} \in S[e_i]$ is the only element of $S[e_i]$, set $\eta(s_{ij})=1$.  Otherwise, for $s_{jk} \in S[e_i]$ there exists $s_{ij} \in e_i S[e_i]$ such that $s_{ik}=s_{ij} \cdot s_{jk}\neq \theta$.  Define $\eta: S\rightarrow D$ by $\eta(s_{jk})=\alpha_{ij}^{-1}[ \xi(s_{ij}, s_{jk} )^{-1}]$.  Then $(\mu, \eta) \in G(S, D)$.

Now a tedious but straightforward calculation using Equations \ref{2cochain} and \ref{twistauto} shows that $(\beta, \zeta)=(\mu, \eta) *(\alpha, \xi)$ restricts to 1 on each $\sim$-block $S[e_i]$.

\end{proof}

For a square-free semigroup $S$ with reduced semigroup $\overline{S}$, Lemma \ref{reduced} shows that $H^2(S, D)=H^2( \overline{S}, D)$.  Therefore, in addition to choosing $[\alpha, \xi] \in H^2 (S, D)$ to be normal, we may also assume that $[\alpha, \xi]$ is trivial on all $\sim$-blocks of $S$.


We now have the following extension of Theorem 2.1 of \cite{Montgomery1} and Theorem 2.2 of \cite{And-D'Amb2}.

\begin{theorem}\label{classification-rlu}
Let $R$ be a locally artinian ring with atomic set $E$.  Then $R$ is a square-free ring if and only if
there is a square-free semigroup $S$, division ring $D$, and a 2-cocycle $(\alpha, \xi) \in Z^2(S, D)$, trivial on all $\sim$-blocks of $S$, such that
$R \cong D^{\alpha}_{\xi} S$.
\end{theorem}

\begin{proof}
Let $R$ be a square-free ring with atomic set $E$ and let $S=S_{\chi}$ for some choice map $\chi: e_iRe_j \rightarrow s_{ij}$, so that $S$ is a square-free semigroup with nonzero idempotent set $E$.  By Theorem \ref{square-free props} (b) and (c), for each $e_i \in E$ there is an isomorphism $\phi_i: D \rightarrow e_iRe_i$ where $D$ is the division ring of $R$.  We will abbreviate 
$$\phi_{i}(d)=d_i$$
for each $e_i \in E$ and $d \in D$.  It follows from Theorem \ref{square-free props} (a) there is a map $\xi \in F^2(S, D)$ such that for all $e_i, e_j, e_k, e_{\ell} \in E$
$$s_{ij} s_{k \ell} \neq 0 \Rightarrow s_{ij} s_{k\ell}= \xi(s_{ij}, s_{jk})_i  s_{i\ell}.$$
Note since $s_{ii}=e_i$ for each $e_i \in E$, we have for all $e_i, e_j \in E$
$$s_{ij} \neq 0 \Rightarrow \xi(e_i, s_{ij})=\xi(s_{ij}, e_j)=1.$$

Next, by Theorem \ref{square-free props} (c),  there exists $\alpha \in F^1(S, \mbox{Aut}(D))$ so that for each $d \in D$ and $s_{ij} \in S^*$ we have
$$(\alpha_{ij}(d))_i s_{ij}=s_{ij} d_j.$$
Since $s_{jj}d_j=e_jd_j=d_{j} e_j=d_j s_{jj}$, we have $\alpha_{j}=1$ for all $j $.

Because multiplication in $R$ is associative,  for $s_{ij} s_{jk} s_{k \ell} \neq 0$ we have
\begin{align*}
& ( s_{ij} s_{jk}) s_{k\ell}= s_{ij}( s_{jk} s_{k\ell})  \Rightarrow  \xi(s_{ij}, s_{jk})_i \xi(s_{ik}, s_{k \ell} )_i  =(\alpha_{ij} (\xi (s_{jk}, s_{k\ell}))_i \xi(s_{ij}, s_{j\ell} )_i
\end{align*}
by equating coefficients of $s_{i\ell}$.

Also, for $d \in D$ and $s_{ij} s_{jk} \neq 0$ we have
\begin{align*}
 ( s_{ij} s_{jk})d_k&=s_{ij} (s_{jk}d_k) 
\Rightarrow  \xi(s_{ij} , s_{jk} )_i s_{ik} d_k=s_{ij} (\alpha_{jk}(d))_j s_{jk} 
\\ 
\Rightarrow & \xi(s_{ij} , s_{jk} )_i  (\alpha_{{ik}} (d))_i s_{ik}=(\alpha_{ij}(\alpha_{jk}(d)))_i \xi(s_{ij}, s_{jk} )_i s_{ik}
\\ 
\Rightarrow & \alpha_{ij} \circ \alpha_{jk}=\tau_{\xi(s_{ij}, s_{jk} )} \circ \alpha_{ik}
\end{align*}
Thus we see that $(\alpha, \xi)$ is a normal 2-cocycle.  The map $\gamma: D_{\xi}^{\alpha} S \rightarrow R$ via $ds_{ij} \rightarrow d_i s_{ij}$ is a clearly a monic ring homomorphism.  But since 
$$R=\bigcup_{u  \in \mathscr{U}_E} uRu \Rightarrow R=\bigcup_{e_i, e_j \in E} e_i Re_j,$$ $\gamma$ is surjective, and $R \cong D^{\alpha}_{\xi} S$.

Conversely, consider $R=D^{\alpha}_{\xi}S $ for some square-free semigroup $S$ with idempotent set $E$, division ring $D$, and $(\alpha, \xi) \in Z^2(S, D)$.  By Lemma \ref{rightiso} and Lemma \ref{reduced}, we may assume $(\alpha, \xi)$ is a normal 2-cocycle whose restriction on each $\sim$-block of $S$ is trivial.  

Note 
$$e_i^2=e_ie_i=\xi(e_i,e_i) e_i\cdot e_i=e_i$$
since $\xi(e_i, e_i)=1$ for all $e_i \in E$.  It follows that $e_iRe_j=D^{\alpha}_{\xi} e_iSe_j$ for all $e_i, e_j \in E$.

Furthermore $e_iRe_i =De_i \Rightarrow e_iRe_ie_i=D e_i$ and for every $e_i$, we may identify the elements of $e_i Re_i$ with elements of $D$.  Hence each $e_iRe_i$ is a division ring.

Since $S$ is square-free, for all $e_i, e_j \in E$, $e_i S^* e_j$ has at most cardinality one, so 
$$\mbox{dim}_{e_iRe_i} (e_i Re_j) \leq 1.$$  
Let $0 \neq e_i x e_j \in e_iRe_j$.  Now $Re_i=Re_iRe_i$ by Theorem \ref{square-free props} (a) and since $\mbox{dim}_{e_iRe_i} (e_i Re_j) = 1$,
$$Re_ixe_j=R \cdot e_i Re_i \cdot e_i x e_j=R \cdot e_iRe_j=Re_iRe_j.$$
So for any $e_j \in E$, the indecomposable projective module $Re_j$ is square-free.  Hence $R$ is left square-free.  A similar result holds for each $e_i R$.  Therefore, $R$ is square-free.
\end{proof}





In order to establish the main theorem of this note, we must consider automorphisms for square-free rings (which will be written as right operators.)  Let $X \subseteq R$ and let $a \in R$.  As in \cite{And-D'Amb2} we set
$$X_{\ell}(a)=\{ x \in X:~xa \neq 0 \},~~~~X_{r}(a)=\{ x \in X: ax \neq 0\},$$
and
$$X(a)=X_{\ell}(a) \cup X_r (a).$$

We say $X$ is \textit{left (right) summable} in case $X_{\ell}(e_i)$ ($X_r (e_i)$) is finite for every $e_i \in E$.  We say $X$ is \textit{summable} in case it is both left and right summable.

Let $X \subseteq R$ be left summable.  Since $X_{\ell} (a)$ is finite, for each $a \in R$ there exists an $R$-homomorphism
$$\lambda_X:R_{R} \rightarrow R_{R}~~\mbox{via}~~(a) \lambda_X = \left( \sum X \right)a=\sum_{x \in X} xa.$$
If $X$ is right summable, there exists an $R$-homomorphism
$$\rho_X:{}_{R} R \rightarrow {}_{R} R~~\mbox{via}~~(a) \rho_X =a \left( \sum X \right)=\sum_{x \in X} ax.$$

Denote the set of all summable sets of $R$ by $\Sum R$.  If $X, Y \in \Sum R$, define the product of $X$ and $Y$ by
$$XY=\{xy:~x \in X, y \in Y\}.$$
By Lemma 1.4 of \cite{And-D'Amb2}, if $X, Y \in \Sum R$, then $XY \in \Sum R$ and it follows that $\Sum R$ is a semigroup.  Define a relation $\equiv$ on $\Sum R$ by
$$X \equiv Y \Leftrightarrow \lambda_X = \lambda_Y.$$  By Lemma 1.5 of \cite{And-D'Amb2}, the relation $\equiv$ is a multiplicative congruence on $\Sum R$ and $\Sum R/\equiv$ is a monoid with multiplicative identity $E$.

Denote the group of units of the monoid $\Sum R/\equiv$ by $U(R)$.  An element $X \in \Sum R$ will be said to be a \textit{unit} or to be \textit{invertible} if its class $X^{\equiv}$ is in $U(R)$.  Thus, $X \in \Sum R$ is invertible iff there is some $Y \in \Sum R$ with
$$XY \equiv YX \equiv E.$$
If such a $Y$ exists, we write $Y=X^{-1}$ and call $Y$ the \textit{inverse} of $X$.  For $X$ invertible with inverse $Y=X^{-1}$ define the map $\tau_X: R \longrightarrow R$ by
$$\tau_X:a \rightarrow (a) \lambda_Y \rho_X.$$
This notation is purposefully similar to the previously defined inner automorphism $\tau_d \in \Inn{D}$.  By Lemma 1.6 of \cite{And-D'Amb2} the map $\tau_X$ is an automorphism of $R$.

We say an automorphism $\gamma \in \Aut R$ is \textit{inner} in case there exist maps
$$u, v: \mathscr{U}_E \longrightarrow R$$
such that for all $e, f \in \mathscr{U}_E$,
$$ u(e) \in(e)\gamma Re,~~~~\mbox{and}~~~~ v(f) \in f R(f)\gamma$$
and for all $x \in eRf$,
$$(x)\gamma=u(e)xv(f).$$
For such an inner automorphism we have
$$u(e) ve(e)=(e)\gamma~~~\mbox{and}~~~v(e)u(e)=e$$
for all $e \in \mathscr{U}_E$.  We can now state an important and useful result (See Proposition 1.7 of \cite{And-D'Amb2}).

\begin{proposition}
An automorphism $\gamma \in \Aut R$ is inner if and only if $\gamma=\tau_X$ for some $X \in U(R)$.
\end{proposition}

For
each $e_i \in E$, its $E$-\textit{block} is the set
$$E[e_i]=\{e_j \in E|~e_j \sim e_i\}=E \cap S[e_i]$$
of all idempotents of $S$ equivalent to $e_i$.  We shall assume that each $E$-block is ordered.  If the set $E[e_i]$ is finite, there is a bijection
$$\omega: E[e_i] \rightarrow \{ 1, \ldots, |E[e_i]| \}$$
so that each element of $E[e_i]$ is labeled with a number between 1 and the cardinality of $E[e_i]$.
If the set $E[e_i]$ is not finite, let $\omega_{e_i}$ be a bijection of $E[e_i]$ onto the first ordinal of cardinality of $E[e_i]$.  For each $e_i \in E$ set $\omega_{i}(\overline{e}_i)$ to be the least element of $\omega_{e_i}(E[e_i])$ so that $\overline{E}=\{ \omega_{i}(\overline{e_i}) |~e_i \in E \}$ is a complete set of $\sim$-representatives.  

For $\phi \in \Aut S$, we have $E^{\phi}=E$ and $E[e_i]^{\phi}=E[e_i^{\phi}]$ for all $e_i \in E$.  We say that $\phi$ is \textit{normal} if
$$\omega_{e_i^{\phi}} (e_i^{\phi})=\omega_{i}(e_i)$$
for all $e_i \in E$.

An automorphism $\gamma \in \Aut R$ of the square-free ring $R=D^{\alpha}_{\xi} S$ is \textit{normal} if for each $e_i \in E$,
$$(e_i) \gamma \in E \hspace{.3cm} \mbox{and} \hspace{.3cm} \omega_{(e_i)\gamma} ((e_i)\gamma)=\omega_i(e_i).$$

For $\gamma \in \Aut{R}$, $(\mathscr{U}_E)\gamma$ is an atomic set of local units for $R$ with atoms $(E)\gamma$.  By the Azumaya-Krull-Schmidt Theorem (see Theorem 12.6 of \cite{frank} and Lemma 1.8 of \cite{And-D'Amb2}), there is an inner automorphism $\tau_X \in \Inn{R}$ such that $ \gamma \tau_X$ is normal. 

Now, as in the artinian case, we write a square-free ring $R$ as $D^{\alpha}_{\xi} S$ for some square-free semigroup $S$ with idempotent set $E$, division ring $D$, and normal $(\alpha, \xi) \in Z^2(S, D)$.  All of the consequences from Section 2 of \cite{Montgomery1} can be extended to square-free rings with local units.  In particular, we have the following.

\begin{theorem} \label{leftiso}
Let $S$ be a square-free semigroup, and let $[\alpha, \xi],~[\beta, \zeta] \in \thin$.  Then
	$$ D^{\beta}_{\zeta} S \cong D^{\alpha}_{\xi} S$$
if and only if there exists an automorphism $\phi \in \mbox{Aut} (S)$ such that
	$$  [\beta, \zeta] = [\alpha^{\phi}, \xi^{\phi}].$$
\end{theorem}


Theorem \ref{leftiso} gives the means to determine when a square-free ring has a related square-free algebra structure.  
If $A=K_{\zeta} S$ is a square-free $K$-algebra and $D$ is some division ring with $\mbox{Cen}(D)=K$, then $D \otimes_K A$ is a square-free ring.  The division ring and square-free semigroup (respectively) of $R=D \otimes_K A$ are easily seen to be isomorphic to $D$ and $S$ (respectively).

Any square-free ring $R=D_{\xi}^{\alpha} S$ isomorphic to a ring of the form $D \otimes_K A$, with $A=K_{\zeta} S$ a square-free $K$-algebra,
is called a \textit{square-free D-algebra}.

We now have the following.

\begin{corollary} \label{algebra}
Let $S$ be a square-free semigroup with a set of idempotents $E$.  Let $D$ be a division ring with $\Cen(D)=K$ and let $(\alpha, \xi) \in Z^2(S, D)$.  Then $R=D^{\alpha}_{\xi}S$ is a $D$-algebra if and only if $[\alpha, \xi]=[1, \zeta]$ for some $\zeta \in Z^2(S, K)$.
\end{corollary}

\begin{proof}
If $[\alpha, \xi]=[1, \zeta]$, then by Theorem \ref{leftiso}, $R=D^{\alpha}_{\xi}S$ is isomorphic to $D^1_{\zeta} S=D_{\zeta} S$.  A straightforward check shows $D \otimes_K K_{\zeta} S \cong D_{\zeta} S$ via the map
$$d \otimes k s_{ij} \longrightarrow dks_{ij},$$
and it follows that $R$ is isomorphic to a $D$-algebra.

For the converse, let $R=D^{\alpha}_{\xi}S$ be a square-free ring, where (by Lemma \ref{rightiso}) $(\alpha, \xi)$ is normal, and suppose $R$ is a square-free $D$-algebra.  Then there is a square-free $K$-algebra $K_{\zeta} S$, with $\zeta \in Z^2(S, K)$, and an isomorphism (written as a right operator) $\gamma: D^{\alpha}_{\xi}S \rightarrow D \otimes_K K_{\zeta} S$.  For each $e_i \in E$, $(1 \otimes e_i)\gamma$ is a primitive idempotent.  So the set $E'=\{ (1 \otimes e_i) \gamma ~|~~e_i \in E\}$ is a set of atoms for $R$ and $\gamma: E \rightarrow E'$ is a bijection with $R(1 \otimes e_i)\gamma \cong Re_i$ for all $e_i \in E$.   By Lemma 1.8 of \cite{And-D'Amb2}, for each $(1 \otimes e_i)\gamma$, there is a summable set $X \subseteq R$ and an inner automorphism $\tau_X  \in \Aut(R)$ with
$$[(1 \otimes e_i)\gamma] \tau_X=e_i. $$
Assume without loss of generality that $(1 \otimes e_i)\gamma=e_i$.
Then for $1 \otimes s_{ij}=(1 \otimes e_i) (1 \otimes s_{ij} ) (1 \otimes e_j)$,
\begin{align*}
(1 \otimes s_{ij} )\gamma=e_i (1 \otimes s_{ij}) \gamma e_j
\in e_iR e_j = Ds_{ij}.
\end{align*}
So for each $s_{ij} \in S$, there is a $d_{ij} \in D$ such that 
$$ (1 \otimes s_{ij} )\gamma=d_{ij} s_{ij}.$$  
Also, since $( d \otimes e_i) \gamma=( 1 \otimes e_i)\gamma( d \otimes e_i)\gamma( d \otimes e_i)\gamma \in e_iRe_i = De_i$, for each $d \in D$ and $e_i \in E$, there is an automorphism
 $$\mu_i : D \rightarrow D$$
 with $(d \otimes e_i) \gamma =\mu_i(d)e_i$.  Then for $s_{ij} \cdot s_{jk}=s_{ik} \in S^*$ we have
\begin{align*}
\mu_i [\zeta(s_{ij}, s_{jk}) ]d_{ik} s_{ik} &=
(\zeta(s_{ij}, s_{jk})  \otimes e_i ) \gamma 
(1 \otimes s_{ik} ) \gamma \\ 
&=
[ ( \zeta(s_{ij}, s_{jk})  \otimes e_i ) (1 \otimes s_{ik} )]\gamma  \\ 
&=[ (1 \otimes s_{ij} ) (1 \otimes s_{jk} )]\gamma=  (1 \otimes s_{ij}) \gamma(1 \otimes s_{jk})\gamma
\\ 
 &=d_{ij} s_{ij} d_{jk} s_{jk}= d_{ij} \alpha_{ij} (d_{jk}) \xi (s_{ij}, s_{jk} ) s_{ik}.
\end{align*}
Equating coefficient of $s_{ik}$ gives
\begin{equation}\label{dequation}
\mu_i [\zeta(s_{ij}, s_{jk}) ] d_{ik}=d_{ij} \alpha_{ij} (d_{jk}) \xi (s_{ij}, s_{jk} ).
\end{equation}

For each $d_{ij}\in D$, let $\tau_{ij} \in \Aut(D)$ be the inner automorphism
$$\tau_{ij}(x)=d_{ij} x d_{ij}^{-1}.$$
For $d \in D$ 
\begin{align*}
d_{ij} s_{ij} \mu_j(d)=d_{ij} s_{ij} \mu_j(d)e_j
&=(1 \otimes s_{ij})\gamma (d \otimes e_j)\gamma \\ 
&= (d \otimes s_{ij})\gamma =(d \otimes e_i )\gamma (1 \otimes s_{ij})\gamma =
\mu_i (d) d_{ij} s_{ij}
\end{align*}
and
\begin{align*}
d_{ij} s_{ij} \mu_j(d)&=
d_{ij} \alpha_{ij}( \mu_j(d) ) s_{ij} \\ 
&=d_{ij} \alpha_{ij}( \mu_j(d) )d_{ij}^{-1} d_{ij}s_{ij}=
\tau_{ij} [ \alpha_{ij} (\mu_j(d) )] d_{ij}s_{ij}.
\end{align*}
Now it follows that
\begin{equation}\label{onemap}
\mu_i(d)=\tau_{ij} [ \alpha_{ij}( \mu_j(d) )]
\end{equation}
for all $d \in D$.
Note that for each $i$, the map $\mu_i$ is an element of $F^0(S, \Aut(D))$. 
Our selection of $d_{ij}$ defines a map $\eta \in F^1 (S, D^*)$ such that
$$\eta(s_{ij})=d_{ij}.$$
Hence, from Equation \ref{dequation} and Equation \ref{onemap} there is $(\mu, \eta) \in G(S, D)$ with
\begin{equation*} 
\mu_i [\zeta(s_{ij}, s_{jk}) ] \eta(s_{ik})=\eta(s_{ij}) \alpha_{ij} (\eta(s_{jk})) \xi (s_{ij}, s_{jk} )
\end{equation*}
and 
\begin{equation*}
\mu_i \circ   \mu_j ^{-1} =\tau_{ij} \circ \alpha_{ij}.
\end{equation*} 
where $\tau_{ij}(x)=d_{ij} xd_{ij}^{-1}=\eta(s_{ij}) x \eta(s_{ij})^{-1}$.   This shows that
$$(\mu, \eta) *(\alpha, \xi)=(1, \zeta) \in Z^2(S, K).$$
\end{proof}

In \cite{And-D'Amb3}, the authors gave a description in terms of a solvable basis to determine when a square-free artinian ring is a $D$-algebra. (See Theorem 4.8 of \cite{And-D'Amb3}.)  By applying Theorem \ref{algebra} to artinian rings, we now have a characterization in cohomological terms of when a square-free ring $R=D_{\xi}^{\alpha}S$ is a $D$-algebra.

\section{The Automorphism Groups of Square-Free Rings}

We now extend the results from \cite{Montgomery1} to square-free rings with local units.  Since much of the argument follows in the exact same way, when possible, we will omit details and refer the interested reader to Section 3 of \cite{Montgomery1}.

Let $R$ be a square-free ring.  By Lemma \ref{reduced} and Theorem \ref{classification-rlu}, we may assume $R = D^{\alpha}_{\xi} S$, where $S$ is a square-free semigroup with idempotent set $E$, $D$ is the canonical division ring of $R$ and $(\alpha, \xi) \in Z^2(S, D)$ is a normal 2-cocycle that is trivial on all $\sim$-blocks of $S$.  We will consider the outer automorphism group
$$\mbox{Out} R  =\Aut R/\Inn R$$
of $R$, where $\mbox{Aut } R$ is the group of all ring automorphisms and $\mbox{Inn }R$ is the normal subgroup of $\mbox{Aut } R$ consisting of maps $\tau_X$ for summable $X \subseteq R$.  As in the previous sections, automorphisms of $R$ and $S$ will be written as right operators.  Automorphisms of $D$ and of $\Aut R$ will be written as left operators.  Finally, $1_D$ and $1_S$, respectively, will each denote the identity automorphism of the division ring $D$ and the semigroup $S$.  An unlabeled $1$ will denote the constant function to the multiplicative identity of $D$ in either $F^1(S, D)$ or $F^2(S, D)$.  

\begin{definition}
Given $(\alpha, \xi) \in Z^2(S, D)$, a pair $(\mu, \eta) \in G(S, D)$ is called an \textit{$(\alpha, \xi)$ 1-cocyle} if
\begin{equation}
\mu_i \circ \alpha_{ij} \circ \mu_j^{-1}= \tau_{\eta(s_{ij})} \circ \alpha_{ij}
\end{equation}
for all $s_{ij} \in S$, and
\begin{equation}
\mu_i[ \xi(s_{ij}, s_{jk})]= \eta(s_{ij}) \alpha_{ij}( \eta(s_{jk})) \xi(s_{ij},s_{jk}) \eta( s_{ij} \cdot s_{jk})^{-1}
\end{equation}
for $(s_{ij}, s_{jk}) \in S^{<2>}$.  The set of all $(\alpha, \xi)$ 1-cocyles forms a subgroup of $G(S,D)$ and will be denoted by 
$$Z^1_{ (\alpha, \xi)} (S, D)=\{ (\mu, \eta) \in G(S, D) |~(\mu, \eta) \mbox{ is an } (\alpha, \xi) \mbox{ 1-cocycle} \}.$$
\end{definition}

There is a group action $\star: F^0 (S, D) \times Z^1_{ (\alpha, \xi)} (S, D) \rightarrow Z^1_{ (\alpha, \xi)} (S, D)$ given by
\begin{align*}
& \nu \star (\mu, \eta)=(\hat{\mu}, \hat{\eta}) 
\end{align*}
where 
\begin{align*}
& (\hat{\mu}_i)(d)= \nu(e_i) \mu_i(d)\nu(e_i)^{-1} \hspace{.5cm}  \mbox{and} \hspace{.5cm}  \hat{\eta}(s_{ij})=\nu(e_i) \eta(s_{ij}) \alpha_{ij} ( \nu(e_j)^{-1})
\end{align*}
for $d \in D$, $s_{ij} \in S$, $\nu \in F^0 (S, D^*)$ and $(\mu, \eta) \in Z^1_{ (\alpha, \xi)} (S, D)$ .

\begin{definition}
The orbit of $(1_D, 1)$ under the action $\star$ is called the set of $(\alpha, \xi)$  \textit{1-coboundaries} and denoted by 
$$B^1_{ (\alpha, \xi)} (S, D)=\{ \nu \star (1_D, 1) |~ \nu \in F^0(S, D) \}.$$  
Observe that $(\mu, \eta) \in B^1_{ (\alpha, \xi)} (S, D)$ if and only if for each $s_{ij}\in S$
$$\mu_i=\tau_{\nu(e_i)} \mbox{ and } \eta(s_{ij})=\nu(e_i) \alpha_{ij}( \nu(e_j)^{-1}).$$
\end{definition}
A quick check shows $B^1_{ (\alpha, \xi)} (S, D)$ is a normal subgroup of $Z^1_{ (\alpha, \xi)} (S, D)$ for any $(\alpha, \xi) \in Z^2(S, D)$.

\begin{definition}
We define the $(\alpha, \xi)$ \textit{1-cohomology group of $S$ with coefficients units in $D$} to be the factor group
$$H^1_{ (\alpha, \xi)} (S, D)=Z^1_{ (\alpha, \xi)} (S, D)/B^1_{ (\alpha, \xi)} (S, D).$$
\end{definition}

For $(\mu, \eta) \in Z^1_{(\alpha, \xi)}(S, D)$, let $\sigma_{\mu \eta}: R \rightarrow R$ be the map defined  by
$$ (ds_{ij})\sigma_{\mu \eta}=\mu_i(d) \eta(s_{ij})s_{ij}$$
for $d \in D$ and $s_{ij}  \in S$.  By extending $\sigma_{\mu \eta}$ to preserve addition and utilizing the defining properties of $Z^1_{(\alpha, \xi)}(S, D)$, one can verify that $\sigma_{\mu \eta}$ is a ring homomorphism.  Then we have the following:

\begin{lemma}
The map
$$ \Lambda: H^1_{(\alpha, \xi)} (S, D) \rightarrow \Out R$$
defined by
$$ \Lambda:( B^1( S, D)) ( \mu, \eta) \rightarrow (\Inn R) \sigma_{\mu \eta}$$
for every $( \mu, \eta) \in Z^1_{(\alpha, \xi)} (S, D)$ is a group monomorphism.
\end{lemma}

\begin{proof}
It is enough to show for  $(\mu, \eta) \in Z^1_{(\alpha, \xi)} (S, D)$ that $(\mu, \eta) \in B^1_{(\alpha, \xi)} (S, D)$ if and only if $\sigma_{\mu \eta} \in \mbox{Inn } R$.  First suppose $(\mu, \eta) \in B^1_{(\alpha, \xi)}(S, D) $.  Then by properties of $B^1_{(\alpha, \xi)}(S, D)$ there exists $\nu \in F^0 (S, D)$ such that 
\begin{align*}
& \mu_i (d)=\tau_{\nu(e_i)} (d)
\end{align*}
and
\begin{align*}
& \eta(s)= \nu(e_i) \alpha_s (\nu(e_j)^{-1})
\end{align*}
for $d \in D$ and $s_{ij}  \in S$.  The set $\nu E=\{ \nu(e_i)e_i| e_i \in E\}$ is summable and invertible.  Since $\alpha_{e_i}=\alpha_i =1_D$, $(\nu E)^{-1}=\{ \nu(e_i)^{-1} e_i |~ e_i \in E \}$.
Then for all $d \in D$ and $s_{ij}\in $S, 
\begin{align*}
( ds_{ij})\sigma_{\mu \eta} =\mu_i(d) \eta(s_{ij}) s_{ij}
 & =\tau_{\nu(e_i)} (d)   \nu(e_i)  \alpha_s ( \nu(e_j)^{-1} )s_{ij} \\
 &=\tau_{\nu(e_i)} (d)   \nu(e_i)  \alpha_s ( \nu(e_j)^{-1} ) s_{ij}   \\ 
&= \nu(e_i)  ds_{ij} \nu(e_j)^{-1}=  \nu(e_i)  d e_i \cdot s_{ij} \cdot \nu(e_j)^{-1} e_j \\ 
&=  \left( \sum \nu E \right) ds   \left( \sum (\nu E)^{-1}  \right) \\ 
&=(ds) \tau_{\nu E} 
\end{align*}
and it follows that $\sigma_{\mu \eta} \in \mbox{Inn } R$.

Now suppose $\sigma_{\mu \eta} \in \mbox{Inn } R$.  Then for all $d \in D^*=D \setminus \{0\}$ and $s_{ij} \in S$, there exists a summable, invertible set $X$ such that $(ds_{ij})\sigma_{\mu \eta}=(ds_{ij})\tau_X$.  By Lemma 1.9 of \cite{And-D'Amb2}, we may assume that there exist $d_{ij}, \hat{d}_{ij} \in D$ for $s_{ij} \in S^*$ with
$$X =\{d_{ij} t_{ij}:~t_{ij} \in S^*\} \hspace{.3cm} \mbox{and} \hspace{.3cm} Y=\{ \hat{d}_{ij} t_{ij} :~t_{ij} \in S^* \}.$$
Then for each $d \in D$ and $s_{ij} \in S^*$
\begin{equation}\label{significant issue 3}
(ds_{ij})\sigma_{\mu \eta}= (ds_{ij})\tau_X   \Rightarrow \mu_i(d) \eta(s_{ij}) s_{ij}= (\sum X) (ds_{ij}) (\sum Y)  .
\end{equation}
Setting $d=1$ in Equation \ref{significant issue 3} and noting
$$ \eta(s_{ij}) s_{ij}= \eta(s_{ij}) e_i \cdot s_{ij} \cdot e_{j}=e_i( \eta(s_{ij}) s_{ij}) e_j$$
gives
\begin{align*}
\eta(s_{ij})s_{ij}= (\sum X) s_{ij} (\sum Y) &=e_i \left(\sum d_{ij} t_{ij} \right)e_i \cdot s_{ij} \cdot fe_j \left( \sum  \hat{d}_{ij} t_{ij} \right) e_j.
\end{align*}
Since $S$ is square free, $e_i \left(\sum d_{ij} t_{ij} \right)e_i =d_{ii} e_i$, $e_j  \left( \sum  \hat{d}_{ij} t_{ij} \right) e_j=\hat{d}_{jj} e_j$ and
\begin{align*}
\eta(s_{ij}) s_{ij}=d_{ii}  s_{ij} \hat{d}_{jj}=d_{ii} \alpha_{ij} (\hat{d}_{jj}) s_{ij}.
\end{align*}
It follows that
\begin{equation}\label{eta relation}
\eta(s_{ij}) = d_{ij} \alpha_{ij} (\hat{d}_{jj}).
\end{equation}

Since $\sigma_{\mu \eta} \in \mbox{Aut }R$ we have $(e_i^2)\sigma_{\mu \eta}=(e_i)\sigma_{\mu \eta}(e_i) \sigma_{\mu \eta}$.  So 
$(e_i)\sigma_{\mu \eta}=(e_i)\sigma_{\mu \eta}(e_i)\sigma_{\mu \eta}$
and $(e_i)\sigma_{\mu \eta}=\eta(e_i) e_i$ implies
\begin{align*}
\eta(e_i) e_i&=( \eta(e_i) e_i) (\eta(e_i) e_i) \\
&=\eta(e_i)  \alpha_{i}( \eta(e_i)) e_i^2 \\
&=\eta(e_i) \eta(e_i) e_i.
\end{align*}
So $\eta(e_i)=1$ and setting $s_{ij}=e_i$ in Equation \ref{eta relation} gives 
$$1=d_{ii} \hat{d}_{ii} \Rightarrow \hat{d}_{ii}=d_{ii}^{-1}.$$
Abbreviate $d_{ii}=d_i$ and define $\nu \in F^0( S, D)$ by
$$ \nu(e_i)=d_i$$
for $e_i \in E$.  By Equation \ref{eta relation}
\begin{equation}\label{inner and boundary1}
\eta(s_{ij})= \nu(e_i) \alpha_{ij} (\nu(e_j)^{-1})
\end{equation}
 and
 \begin{align*}
\mu_i(d) \eta(s_{ij}) s_{ij}= (\sum X) (ds_{ij}) (\sum Y) \Rightarrow \mu_i(d) d_i \alpha_{ij} (d_j^{-1})s_{ij}= d_i d \alpha_{ij} (d_j^{-1}) s_{ij}
\end{align*}
from which
\begin{equation}\label{inner and boundary2}
\mu_i(d)=d_i d d_i^{-1}= \tau_{\nu(e_i)} (d).
\end{equation}
It now follows from Equations \ref{inner and boundary1} and \ref{inner and boundary2} that $( \mu, \eta) \in B^1 (S, D)$.
\end{proof}

Recall, an automorphism $\gamma \in \Aut R$ of the square-free ring $R=D^{\alpha}_{\xi} S$ is \textit{normal} if for each $e_i \in E$,
$$(e_i) \gamma \in E \hspace{.3cm} \mbox{and} \hspace{.3cm} \omega_{(e_i)\gamma} ((e_i)\gamma)=\omega_{i}(e_i).$$
The set of normal automorphisms  of $R$ forms a subgroup of $\Aut R$, which we denote by $\mbox{Aut}_0 R$.

Extending Lemma 2.2(a) of \cite{Montgomery1}, for each $\gamma \in \Aut_0 R$ there is an automorphism $\phi_{\gamma} \in \Aut S$
$$s^{\phi_{\gamma}}= \begin{cases} (e_i)\gamma \cdot S \cdot (e_j)\gamma,  & \quad \mbox{if} \quad s=e_i \cdot s \cdot e_j \neq \theta \vspace{.2cm} 
\cr  \theta & \quad  \mbox{if} \quad s=\theta.
\end{cases} $$

Define $\Phi_0: \Aut_0 R \longrightarrow \Aut_0 S$ by $\Phi_0 (\gamma)=\phi_{\gamma}$.  
We claim that $\Inn R \cap \Aut_0 R \subseteq \Ker \Phi_0$.  Let $X \in \Sum R$ be invertible with $Y=X^{-1}$.  Assume that $\tau_X \in \Aut_0 R$ and let $e_i \in E$.  Since $\tau_X$ is normal, $e_j=(e_i)\tau_X \in E$.  

Set 
\begin{align*}
y=e_j((e_i)\lambda_Y) \hspace{.5cm} \mbox{and} \hspace{.5cm} x=((e_i)\rho_X)e_j.
\end{align*}
Since $(a) \tau_X=(a) \lambda_Y \rho_X$ we have $yx= e_j ( e_i) \tau_X e_j=e_j$ and $xy= e_i ( e_j) \tau_Y e_i=e_i$ so that $Re_i \cong Re_j$ by Lemma 1.8 of \cite{And-D'Amb2}.  Since $\tau_X$ is normal, it follows that $e_i=e_j$ and $\Phi_0 (\tau_X)$ is the identity on $E$.  Then for each $s_{ij} \in S^*$, 
$$(s_{ij})\Phi_0 [\tau_X]=e_i \cdot S \cdot e_j =s_{ij}$$
and $\Phi_0 (\tau_X)=1_S$.
Next, let $\gamma \in \Aut{R}$.  Then $(\mathscr{U}_E)\gamma$ is an atomic set of local units for $R$ with atoms $(E)\gamma$.  By the Azumaya-Krull-Schmidt Theorem (see Theorem 12.6 of \cite{frank} and Lemma 1.8 of \cite{And-D'Amb2}), there is $\tau_X \in \Inn{R}$ such that $ \gamma \tau_X$ is normal.  
Therefore, there is a map
$$ \Phi: \Out R \longrightarrow \Aut S$$
given by
$$ \Phi: (\Out R) \gamma \rightarrow \phi_{\gamma} $$
for all $\gamma \in \Aut_0 R$.  A straightforward check shows that $\Phi$ is a group homomorphism.  The same argument as in Lemma 3.7 of \cite{Montgomery1} now applies and we have

\begin{align*}
1 \longrightarrow H^1_{(\alpha, \xi)} (S, D) \overset{\Lambda}{\longrightarrow } \Out R
\overset{\Phi}{\longrightarrow }\Aut S
\end{align*}
an exact sequence.

Recall that $\mbox{Aut }S$ acts on $\thin$.  For any element $[\alpha, \xi] \in \thin$, the stabilizer of the element $[\alpha, \xi] \in H^2(S, D)$ is the set
$$\mbox{Stab}_{ [\alpha, \xi]}(\mbox{Aut }S)= \{ \phi \in \mbox{Aut}(S)| ~~[\alpha, \xi]= [\alpha^{\phi}, \xi^{\phi}]\}.$$
That is, $\phi \in  \mbox{Stab}_{ [\alpha, \xi]}(\mbox{Aut }S)$ if and only if there exists an element $(\mu, \eta) \in G(S, D)$ such that

\begin{equation}
\mu_i \circ \alpha_{ij} \circ \mu_j^{-1}= \tau_{\eta(s_{ij})} \circ \alpha_{ij}^{\phi}
\end{equation}
and
\begin{equation}
\mu_i[ \xi(s_{ij}, s_{jk})]= \eta(s_{ij}) \alpha^{\phi}_{ij}( \eta(s_{jk})) \xi^{\phi}(s_{ij},s_{jk}) \eta( s_{ij} \cdot s_{jk})^{-1}
\end{equation}
for all $s_{ij} \cdot s_{jk}  \in S^*$.

The remaining arguments leading up to our main theorem now carry over, essentially verbatim, from Section 3 of \cite{Montgomery1}.  So we have the following.

\begin{theorem}\label{SES}
Let $R=D^{\alpha}_{\xi} S$ be a square-free ring with associated semigroup $S$, division ring $D$ and $(\alpha, \xi) \in Z^2(S, D)$.  The following sequence is exact:

\begin{align*}
1 \longrightarrow H^1_{(\alpha, \xi)} (S, D) \overset{\Lambda}{\longrightarrow } \Out R
\overset{\Phi}{\longrightarrow } \Stab {}_{(\alpha, \xi)} (\Aut S) \longrightarrow 1.
\end{align*}
If $[\alpha, \xi]=[1_D, 1] \in H^2(S, D)$, then $\Stab{}_{(\alpha, \xi)}(\Aut S)=\Aut S$ and the following sequence is exact:
\begin{align*}
1 \longrightarrow H^1_{(\alpha, \xi)} (S, D) \overset{\Lambda}{\longrightarrow } \Out R
\overset{\Phi}{\longrightarrow } \Aut S \longrightarrow 1.
\end{align*}
Moreover, this last sequence splits, and so in this case $\Out R$ is isomorphic to 
$$H^1_{(\alpha, \xi)}(S, D) \ltimes \Aut S.$$
\end{theorem}

\begin{proof}
The details of this proof are now identical to those of Theorem 3.9 of \cite{Montgomery1}.
\end{proof}

Applications of Theorem \ref{SES} to artinian rings can be found in \cite{And-D'Amb} and\cite{Montgomery1}.  In particular, the semigroup $S$ of Example 3.12 in \cite{Montgomery1} can easily be adapted to that of a (not necessarily finite) regular semigroup for a ring with local units.  From this example, and Corollary \ref{algebra}, we have a class of square-free rings satisfying Theorem \ref{SES} not described in \cite{And-D'Amb2}.

\end{document}